\documentclass[12pt]{amsart}

\usepackage{amssymb}  %required for \subsetneqq
\usepackage{amsxtra}    %required for \sphat\, \spcheck\,
\usepackage{graphicx,color}
\usepackage[mathscr]{eucal}

 \RequirePackage{geometry}
\usepackage{verbatim} %required for \begin{comment} \end{comment}

\allowdisplaybreaks

\newcommand{\al}{\alpha}
\newcommand{\be}{\beta}
\newcommand{\ga}{\gamma}

\newcommand{\de}{\delta}
\newcommand{\De}{\Delta}

\newcommand{\ve}{\varepsilon}

\newcommand{\si}{\sigma}

\newcommand{\vp}{\varphi}

\newcommand{\rn}{\mathbb R^n}

\newcommand{\lp}{L^{p}}

\newcommand{\li}{L^{\infty}}

\newcommand{\wt}{\widetilde}
\newcommand{\wh}{\widehat}

\newcommand{\p}{\partial}
\newcommand{\real}{\textup{Re}\, }

\newcommand{\f}{\frac}

\newcommand{\tf}{\tfrac}

\newcommand{\nf}{\infty}
\newcommand{\q}{\quad}

\newcommand{\set}[1]{\left\{#1\right\}}
\newcommand{\abs}[1]{\left\vert #1\right\vert}%
\newcommand{\norm}[1]{\left\Vert #1\right\Vert}%
\newcommand{\trinorm}[1]{\left\|\kern-1pt\left|#1\right|\kern-1pt\right\|}%

\DeclareMathOperator{\supp}{\mathrm{supp}}

\theoremstyle{definition}
\newtheorem{defn}{Definition}[section]

\theoremstyle{theorem}
\newtheorem{thm}[defn]{Theorem}
\newtheorem{prop}[defn]{Proposition}

\newtheorem{lem}[defn]{Lemma}
\theoremstyle{remark}

\newcommand{\bbr}{\mathbb R}
\newcommand{\qq}{\qquad}

\newcommand{\ep}{\ve}
\newcommand{\beq}{\begin{equation}}
\newcommand{\eeq}{\end{equation}}

\begin{document}

\title[H\"ormander Multiplier Theorem]{The H\"ormander Multiplier Theorem I: The Linear Case Revisited}

\author[Grafakos]{Loukas Grafakos}
\address{Department of Mathematics, University of Missouri, Columbia, MO 65211, USA}
\email{grafakosl@missouri.edu}

\author[He]{Danqing He}
\address{Department of Mathematics, 
Sun Yat-sen (Zhongshan) University, 
Guangzhou, 510275, 
P.R. China}
\email{hedanqing35@gmail.com}

\author[Honzik]{Petr Honzik}
\address{Faculty of Mathematics and Physics, Charles University in Prague,
Ke Karlovu 3,
121 16 Praha 2, Czech Republic}
\email{honzik@gmail.com}

\author[Nguyen]{Hanh Van Nguyen}
\address{Department of Mathematics,
The University of Alabama, 
Box 870350,
Tuscaloosa, AL 35487-0350}
\email{hanhnguyenvan@gmail.com}
\thanks{The first author was supported by the Simons Foundation. The third author
 was supported by the ERC CZ grant LL1203 of the Czech Ministry of Education}
\thanks{2010 Mathematics Classification Number 42B20, 42B99}

\begin{abstract}
 We discuss
$L^p(\rn)$ boundedness for Fourier multiplier operators that satisfy the hypotheses of the
H\"ormander multiplier theorem in terms of
 an optimal condition that relates the distance $|\f 1p-\f12|$ to the smoothness $s$ of
the associated multiplier measured in some Sobolev norm. We provide new counterexamples to justify the 
optimality of the condition $|\f 1p-\f12|<\frac sn$ and we   discuss  the 
endpoint case $|\f 1p-\f12|=\f sn$.
\end{abstract}

\maketitle

\section{Introduction}

To a bounded function $\si$ on $\rn$ we associate a linear multiplier operator
$$
T_\si(f)(x) = \int_{\rn} \wh{f}(\xi) \si(\xi) e^{2\pi i x\cdot \xi}d\xi
$$
where $f$ is a Schwartz function on $\rn$ and
$\wh{f}(\xi) = \int_{\rn} f(x)   e^{-2\pi i x\cdot \xi}dx$ is its Fourier transform. The classical theorem of
Mihklin~\cite{Mikhlin} states that if the condition
\begin{equation}\label{10}
|\partial^\alpha \si(\xi)|\leq C_\alpha |\xi|^{-| \alpha|}
\end{equation}
holds for all { multi-indices} $\al$ with size $|\al | \le [n/2]+1$, then $T_\si$ admits a bounded extension from
$L^p(\rn)$ to itself for all $1<p<\nf$.

Mikhlin's theorem was extended by H\"ormander~\cite{Hoe} to multipliers with fractional derivatives in some $L^r$ space.
To precisely describe this extension, let $\De$ be the Laplacian, let $(I-\De)^{s/2} $ denote the operator given on the Fourier transform by multiplication by
$(1+4\pi^2 |\xi|^2)^{s/2}$ and  for $s>0$, and
let $L^r_s$ be the standard  Sobolev  space of all functions $h$ on $\rn$
with norm
$$
\|h\|_{L^r_\gamma}:=\|(I-\Delta)^{s/2} h \|_{L^r} <\nf.
$$
Let
$\Psi$ be a Schwartz function whose Fourier transform is supported in the annulus of the form
$\{\xi: 1/2< |\xi|< 2\}$ which satisfies $\sum_{j\in \mathbb Z} \wh{\Psi}(2^{-j}\xi)=1$ for all $\xi\neq 0$.

H\"ormander's extension of Mikhlin's theorem says that if    $1< r\le 2$ and  $s>n/r$,  a bounded function $\si$ satisfies
\begin{equation}\label{2}
\sup_{k\in \mathbb Z} \big\|\wh{\Psi}\si (2^k \cdot)\big\|_{L^r_s}<\infty,
\end{equation}
 i.e., $\si$ is uniformly (over all dyadic annuli) in the Sobolev space $L^r_s$,   
then $T_\si$ admits a bounded extension from
$L^p(\rn)$ to itself for all $1<p<\nf$, and is also of weak type $(1,1)$.  An endpoint result for this
multiplier theorem involving a Besov space  was given by Seeger~\cite{Seeger}.
  The least number of derivatives  imposed on  the multiplier in H\"ormander's  condition~\eqref{2}
is when $r=2$. In this case, under the assumption of  $n/2+\ve$ derivatives in $L^2$
uniformly (over all dyadic annuli),  we obtain  boundedness of $T_\si$ on $L^p(\rn)$
for all $p \in (1,\nf)$.    It is natural to ask
whether $L^p$ boundedness holds for some   $p$  if $s<n/2$.

Calder\'on and Torchinsky~\cite{CT} used an interpolation technique to prove  that if \eqref{2} holds,
then   the multiplier operator $T_\si$ is bounded from $L^p(\bbr^n)$ to itself whenever $p$ satisfies
\beq\label{1c}
\Big| \f 1p -\f 12 \Big| <\f sn
 \eeq
 and
\beq\label{2c}
 \Big| \f 1p -\f 12 \Big| \le \f 1r\, .
 \eeq

It is not hard to verify  that if $\si$ satisfies \eqref{2} and $T_\si$ is bounded from $L^p(\rn)$ to itself, then we must
necessarily have $rs\ge n$; see Proposition~\ref{Nec}. Thus  $\f 1r\le \f sn$ and comparing conditions  \eqref{1c} and  \eqref{2c}
we notice that \eqref{2c}  restricts \eqref{1c}. On the other hand, if we only have conditions \eqref{2} and \eqref{1c} for some $r,s$ with $rs>n$, $r\in (1,\nf)$, $s\in (0,\nf)$,
then one can find an $r_o$ such that $|\f 1p -\f 12|\le \f1{r_o}<\f sn$ and
$r_o<r$. In view  of standard embeddings between Sobolev spaces\footnote{This could be proved via the Kato-Ponce inequality  $\|FG\|_{L^q_s}\le C \|F\|_{L^{q_1}_s}
\|G\|_{L^{q_2}_s} $,  $1/q=1/q_1+1/q_2$  with $q=r_o$ and $q_1=r$; see~\cite{KP}, \cite{GO}.} we obtain that
\begin{equation}\label{nj}
\sup_{k\in \mathbb Z} \big\|\wh{\Psi}\si (2^k \cdot)\big\|_{L^{r_o}_s}\le C\,
\sup_{k\in \mathbb Z} \big\|\wh{\Psi}\si (2^k \cdot)\big\|_{L^{r}_s} <\nf\, ,
\end{equation}
and thus we can deduce the boundedness of $T_\si$ on $L^p(\rn)$ by the
aforementioned Calder\'on and Torchinsky~\cite{CT} result using the space $L^{r_o}_s$. So  assumption \eqref{2c} is not necessary.

In this note we provide a   self-contained proof of the
  $L^p$ boundedness of $T_\si$ only under assumption \eqref{1c}. Moreover, we
show that \eqref{1c} is optimal in the sense that within the class of multipliers $\si$
for which \eqref{2} holds,  if $T_\si$ is bounded from $L^p$ to itself, then
we must necessarily have  $|\f 1p -\f 12|\le \f sn$.   Theorem~\ref{ThmMain}
is mostly folklore, and could be proved via the interpolation result of Connett and Schwartz \cite{CS}, but
here we provide a ``bare hands" proof. The counterexamples we supply
(Section 4)    seem to be new.

\begin{thm}\label{ThmMain}
  Fix $1< r<\nf $ and $0<s\le \f n2$ such that $rs>n$. Assume that \eqref{2} holds.
Then   $T_\sigma$ maps $L^p(\mathbb R^n)$ to $L^p(\mathbb R^n) $ for all $p\in (1,\infty)$ such that
$\big|{\frac1p-\frac12}\big|<\frac{s}{n} $. Moreover, if  $T_\si$ is  bounded from $L^p(\rn)$ to itself for all $\si$
such that  \eqref{2} holds, then we must have $\big|{\frac1p-\frac12}\big|\le \frac{s}{n} $. 
\end{thm}

We note that  the strict inequality in condition $rs>n$ is necessary as there exist unbounded functions in $L^r_{n/r}(\rn)$,  while  multipliers are always in $\li$.

On the critical case $\big|\f1p-\f12 \big|= \f sn$, $1\le p<2$, there are  positive results   for $1<p<2$ (see Seeger~\cite{Seeger2})   and  for $p=1$ by Seeger~\cite{Seeger}. In Section 5 we discuss a direct way to relate the  results in the cases $p=1$ and $1<p<2$ via direct interpolation  that yields the   following result as a consequence of  the main theorem in \cite{Seeger}:

\begin{prop}[\cite{Seeger2}]\label{KKLLMM}
Given $0\le s\le \f n2$, $1< p< 2$ satisfy   $\big|\f1p-\f12 \big|= \f sn$, then we have 
$$
\big\| T_\si \big\|_{L^p\to L^{p,2}} \le C\, \sup_{k\in \mathbb Z}  \| \si (2^k \,\cdot\,) \wh\Psi \|_{B_{ \f ns}^{s,1}}\, .
$$
\end{prop}
\noindent Here    $L^{p,2}$ denotes the Lorentz space of functions $f$ for which $t^{1/p} f^*(t)$ lies in $L^2((0,\infty), \f{dt}{t})$, where $f^*$ is the 
nondecreasing rearrangement of $f$;   
for the definition of the Besov space $B_{ \f ns}^{s,1}$ see Section 5. Other types   
 of endpoint results   involving
$L^p$ norms as opposed to $L^{p,2}$ norms were provided by Seeger~\cite{Seeger3}. 

\section{Complex interpolation}

This section contains an
interpolation result proved in a simpler way than that of Calder\'on and Torchinsky~\cite{CT}.  We  denote by $S$ the strip in the complex plane with $0<\Re(z)<1.$

\begin{lem}\label{lem:016}
  Let $0< p_0<p<p_1<\infty$ be related as in $1/p=(1-\theta)/p_0+\theta/p_1$ for some $\theta \in (0,1)$. 
  Given $f\in {\mathscr C}_0^\nf(\mathbb R^n)$ and   $\ve>0,$ there exist  
  smooth functions $h_j^\ve$, $j=1,\dots, N_\ve$, supported in  cubes   on $\mathbb R^n$ with pairwise disjoint interiors,   
  and nonzero complex constants $c_j^\ve$ such that  the functions
  $$
  f_z^\ve  =  \sum_{j=1}^{N_\ve} |c_j^\ve|^{\frac p{p_0} (1-z) +   \frac p{p_1}  z}  \, h_j^\ve
  $$
satisfy
$$
 \big\|{f_\theta^\ve-f}\big\|_{L^2}+ \big\|{f_\theta^\ve-f}\big\|_{L^{p_0}}^{\min(1,p_0)}+ \big\|{f_\theta^\ve-f}\big\|_{L^{p_1}} ^{\min(1,p_1)}<  \ve  
 $$
and 
$$
  \|{f_{it}^\ve}\|_{L^{p_0}}^{p_0} \le   \|f \|_{L^p}^p +\ve'  \, , \q
  \|{f_{1+it}^\ve}\|_{L^{p_1}}^{p_1} \le     \|f \|_{L^p}^p  +\ve'\, , 
$$
where $\ve'$ depends on $\ve,p_0,p_1,p, \|f\|_{L^p}$ and tends to zero as $\ve\to 0$. 
\end{lem}

\begin{proof}
%We only prove the lemma when $p\ge 1$ as the other case is similar.
Given $f\in {\mathscr C}_0^\nf(\mathbb R^n)$  and   $ \ve>0$, by   uniform continuity there
there are $N_\ve$   cubes $Q_j^\ve$ (with disjoint interiors) and nonnegative  constants $c_j^\ve $ such that
$$
\Big\| f - \sum_{j=1}^{N_\ve} c_j^\ve \chi_{Q_j^\ve} \Big\|_{L^2} +
\Big\| f - \sum_{j=1}^{N_\ve} c_j^\ve \chi_{Q_j^\ve} \Big\|_{L^{p_0}} ^{\min(1,p_0)}   
+\Big\| f - \sum_{j=1}^{N_\ve} c_j^\ve \chi_{Q_j^\ve} \Big\|_{L^{p_1}} ^{\min(1,p_1)}  <\ve    \, .  %^{\max(1/p,1)}
$$

 Find nonnegative smooth functions   $g_j^\ve\le \chi_{Q_j^\ve}$ such that
$$
\Big\| \sum_{j=1}^{N_\ve} c_j^\ve (  g_j^\ve-    \chi_{Q_j^\ve}   )  \Big\|_{L^2} +
\Big\|  \sum_{j=1}^{N_\ve}  c_j^\ve ( g_j^\ve-    \chi_{Q_j^\ve}   ) \Big\|_{L^{p_0}} ^{\min(1,p_0)}+
\Big\|  \sum_{j=1}^{N_\ve}  c_j^\ve ( g_j^\ve-    \chi_{Q_j^\ve}   ) \Big\|_{L^{p_1}}^{\min(1,p_1)} <\ve   
$$
and
$$
\bigg( \sum_{j=1}^{N_\ve} |c_j^\ve|^p \| g_j^\ve - \chi_{Q_j^\ve}\|_{L^{p_0}}^{p_0} \bigg)^{\f 1{p_0}}+
\bigg( \sum_{j=1}^{N_\ve} |c_j^\ve|^p \| g_j^\ve - \chi_{Q_j^\ve}\|_{L^{p_1}}^{p_1} \bigg)^{\f 1{p_1}} <\ve\, . 
$$
Let $\phi_j^\ve$ be the argument of the complex number $c_j^\ve$. 
Set $h_j^\ve = e^{i\phi_j^\ve} g_j^\ve$ and
notice that 
$ f_\theta^\ve = \sum_{j=1}^{N_\ve} |c_j^\ve| h_j^\ve =  \sum_{j=1}^{N_\ve}  c_j^\ve  g_j^\ve $  satisfies 
$$
 \big\|{f_\theta^\ve-f}\big\|_{L^2}+ \big\|{f_\theta^\ve-f}\big\|_{L^{p_0}}^{\min(1,p_0)}+ \big\|{f_\theta^\ve-f}\big\|_{L^{p_1}} ^{\min(1,p_1)}<  \ve .
 $$
Moreover, the choice of $g_j^\ve$   implies that
$$
\| f_{it} \|_{L^{p_0}} \le \big( B^{\min(1,p_0)} +\ve ^{\min(1,p_0)} \big)^{\f{1}{\min(1,p_0)} }, 
$$
 where 
 $$
 B=   \Big\| \sum_{j=1}^{N_\ve} c_j^\ve \chi_{Q_j^\ve} \Big\|_{L^p}^{\f{p}{p_0}} \le 
\Big( \big(\ve^{\min(1,p )} +\|f\|_{L^p}^{\min(1,p )}  \big)^{ \f{1}{\min(1,p )} }  \Big)^{\f{p}{p_0}}.
 $$
 An analogous estimate holds for   $f_{1+it}$. 
 Given $a,c>0$ and $\ve>0$ set $\ve'=\ve'(\ve,a,c)= (\ve^a+c^a)^{1/a}-c$. Then $ (\ve^a+c^a)^{1/a} \le \ve'+c$ and $\ve'\to 0 $ as $\ve\to 0$. Then for a suitable $\ve'$ that only depends on 
 $\ve, p,p_0,p_1, \|f\|_{L^p}$, the preceding estimates give:
  $ \|{f_{it}^\ve}\|_{L^{p_0}}^{p_0} \le   \|f \|_{L^p}^p +\ve'$ and $ \|{f_{1+it}^\ve}\|_{L^{p_1}}^{p_1} \le   \|f \|_{L^p}^p +\ve'$, as claimed. 
\end{proof}

\begin{lem}\label{lem:analint}
For $z$ in the strip $a<\Re(z)<b$ and $x\in \mathbb R^n$,
let   $H(z,x)$ be analytic in $z$ and  smooth in $x\in \mathbb{R}^n$ that satisfies
  $$
  \abs{H(z,x)}+\abs{\f{dH}{dz} (z,x)}\le H_*(x),\quad\forall a<\Re(z)<b,
  $$
 where $H_*$ is  a measurable function on $\mathbb R^n$.
Let $f$ be a  complex-valued smooth function on $\mathbb R^n$ such that
  $$
  \int_{\mathbb R^n}\max\set{\abs{f(x)}^a,\abs{f(x)}^b}\Big\{ 1+ \abs{\log(\abs{f(x)})} \Big\} {H_*(x)}\;dx <\infty .
  $$
  Then the function
  $$
  G(z) = \int_{\mathbb R^n}\abs{f(x)}^ze^{i \textup{Arg } f(x)}H(z,x)dx
  $$
  is analytic on the strip $a<\Re(z)<b$ and continuous up to the boundary.
\end{lem}
\begin{proof}
  Let $A = \set{x\ :\ f(x)\ne 0}.$ For $x\in A$ denote
  $$
  F(z,x) = \abs{f(x)}^ze^{i \textup{Arg } f(x)} H(z,x).
  $$
   Fix $a<\Re(z_0)<b $ and $x\in A$. Then
  \begin{align*}
 & \lim_{z\to z_0}\dfrac{F(z,x)-F(z_0, x)}{z-z_0} \\
 & = \abs{f(x)}^{z_0}\log \abs{f(x)} \,\, e^{i \textup{Arg }  f(x)}
  H(z_0,x) + \abs{f(x)}^{z_0} e^{i \textup{Arg } f(x)}  \f{dH }{dz}(z_0,x)
\end{align*}
  for all $x\in A.$ We also have
  $$
  \abs{\dfrac{F(z,x)-F(z_0, x)}{z-z_0}}\le
  \max\set{\abs{f(x)}^a,\abs{f(x)}^b}\Big( 1+ \abs{\log \abs{f(x)} }  \Big)  H_*(x)
  %+\max\{\abs{f(x)}^a,\abs{f(x)}^b\}H_*(x)
  $$
  for all $x\in A.$ By Lebesgue dominated convergence theorem, the function $G$ is analytic and its derivative is
  $$
  G'(z)= \int_{\mathbb R^n} \bigg[ \abs{f(x)}^{z }\log(\abs{f(x)})e^{i \textup{Arg }  f(x)}
  H(z ,x) + \abs{f(x)}^{z } e^{i \textup{Arg } f(x)}  \f{dH }{dz}(z ,x)   \bigg]  dx
  $$
Also, the function $G$ is also continuous on the boundary $\Re(z)=a$ and $\Re(z) = b.$
\end{proof}
\begin{lem}[{ \cite{CFA, hirschman}}]\label{lem:ThreeLines} %[Modern Fourier Analysis - Grafakos]
  Let $F$ be analytic on the open strip $S=\set{z\in\mathbb{C}\ :\ 0<\Re(z)<1}$ and continuous on its closure.
  Assume that for all  $0\le \tau \le 1$ there exist functions $A_\tau$ on the { real} line such that
$$
    | F(\tau+it) |  \le A_\tau(t)    \qquad \textup{  for all $t\in\mathbb{R}$,}
$$
and suppose that there exist constants   $A>0$ and $0<a<\pi$ such that for all $t\in \mathbb R$ we have
  $$
0<  A_\tau(t) \le \exp \big\{ A e^{a |t|} \big\} \, .
  $$
 Then   for  $0<\theta<1 $   we have
  $$
  \abs{F(\theta )}\le \exp\left\{
  \dfrac{\sin(\pi \theta)}{2}\int_{-\infty}^{\infty}\left[
  \dfrac{\log |A_0(t ) | }{\cosh(\pi t)-\cos(\pi\theta)}
  +
  \dfrac{\log | A_1(t )| }{\cosh(\pi t)+\cos(\pi\theta)}
  \right]dt
  \right\}\, .
  $$
\end{lem}
In calculations it is crucial to note that
\begin{equation}\label{ide}
\dfrac{\sin(\pi \theta)}{2}\int_{-\infty}^{\infty}
  \dfrac{dt }{\cosh(\pi t)-\cos(\pi\theta)}  =1-\theta\, , \q
  \dfrac{\sin(\pi \theta)}{2}\int_{-\infty}^{\infty}
  \dfrac{dt }{\cosh(\pi t)+\cos(\pi\theta)} = \theta.
\end{equation}

The main result of this section is the following:
%%%%%%%%%%%%%%%%%%%%%%%%%%%%%%%%%
\begin{thm}\label{interpL}
 Fix $1<q_0,q_1,r_0,r_1<\infty$, $0<p_0,p_1,s_0,s_1<\infty$. Suppose that $r_0s_0>n$ and $r_1s_1>n$.
 Let $\wh{\Psi}$ be supported in the annulus $1/2\le |\xi|\le 2$ on $\mathbb R^n$ and satisfy
 $$
 \sum_{j\in \mathbb Z} \widehat{\Psi}(2^{-j}\xi) = 1, \q\q \xi\neq 0.
 $$
 Assume that for $k\in\set{0,1}$ we have
  \begin{equation}
    \norm{T_\sigma(f)}_{L^{q_k}}\le K_k\sup_{j\in\mathbb Z}\big\|{\sigma(2^j\cdot)\widehat{\Psi}}\big\|_{L^{r_{k}}_{s_k}}
    \norm{f}_{L^{p_k}}
  \end{equation}
  for   all $f\in \mathscr C_0^\nf(\mathbb R^n)$. % and $\sigma$ which make the right hand side finite.
For $0<\theta<1$ let
  $$
  \frac 1p  =\frac{1-\theta}{p_0} + \frac{ \theta}{p_1} , \quad
    \frac 1q =\frac{1-\theta}{q_0} + \frac{ \theta}{q_1} , \quad
       \frac 1r =\frac{1-\theta}{r_0} + \frac{ \theta}{r_1} , \quad
          s =(1-\theta)s_0 +   \theta s_1  .
  $$
Then  there is  a constant $C_*= C_*( r_0,r_1,s_0,s_1,  n) $ such that for all $f\in \mathscr C_0^\nf(\mathbb R^n)$
we have
  \begin{equation}
    \norm{T_\sigma(f)}_{L^{q}(\rn)}\le C_* K_0^{1-\theta} K_1^{\theta} \sup_{j\in\mathbb Z}\big\|{\sigma(2^j\cdot)\widehat{\Psi}}\big\|_{L^{r}_{s}}
    \norm{f}_{L^{p}(\rn)}.
  \end{equation}
\end{thm}
\begin{proof}
  Fix $\widehat{\Phi}$ such that $\supp(\Phi)\subset\set{\frac14\le \abs{\xi}\le 4}$ and $\wh{\Phi}\equiv 1$ on the support of the function $\widehat{\Psi}.$ Denote
  $
  \varphi_j = (I-\Delta)^{\frac{s}2}[\sigma(2^j\cdot)\widehat{\Psi}]
  $
  and define
\begin{equation}\label{100}
  \sigma_z(\xi) = \sum_{j\in \mathbb Z}(I-\Delta)^{-\frac{s_0(1-z)+s_1 z}{2}}
  \left[
 |\varphi_j|^{r(\frac{1-z}{r_0}+\frac{z}{r_1})}e^{i \textup{Arg } (\varphi_j)}
  \right](2^{-j}\xi)\widehat{\Phi}(2^{-j}\xi).
\end{equation}
 This sum has only finitely many terms and we estimate its $\li$ norm.
Fix $\xi \in \rn$. Then there is a $j_0$ such that $|\xi | \approx 2^{j_0}$ and there are only two terms in the sum in
 \eqref{100}.  For these terms we estimate  the $\li$ norm of
 $(I-\Delta)^{-\frac{s_0(1-z)+s_1 z}{2}}
  \big[ |\varphi_j|^{r(\frac{1-z}{r_0}+\frac{z}{r_1})}e^{i \textup{Arg } (\varphi_j)} \big]$.
  For   $z=\tau+it$ with $0\le \tau\le 1$, let $s_\tau= (1-\tau)s_0+\tau s_1$ and $1/r_\tau = (1-\tau)/r_0+\tau /r_1$.
  By the Sobolev embedding theorem
  we have
\begin{align}
\  \Big\| (I-\Delta)^{-\frac{s_0(1-z)+s_1 z}{2}}&
  \big[ |\varphi_j|^{r(\frac{1-z}{r_0}+\frac{z}{r_1})}e^{i \textup{Arg } (\varphi_j)} \big]  \Big\|_{\li}\notag\\
\le  &\  C(r_\tau ,{s_\tau},n)
\Big\| (I-\Delta)^{-\frac{s_0(1-z)+s_1 z}{2}}
  \big[ |\varphi_j|^{r(\frac{1-z}{r_0}+\frac{z}{r_1})}e^{i \textup{Arg } (\varphi_j)} \big]  \Big\|_{ L^{r_\tau}_{s_\tau}  }\notag \\
  \le  &\  C(r_\tau ,{s_\tau},n)\Big\| (I-\Delta)^{it \frac{s_0-s_1 }{2}}
  \big[ |\varphi_j|^{r(\frac{1-z}{r_0}+\frac{z}{r_1})}e^{i \textup{Arg } (\varphi_j)} \big]  \Big\|_{ L^{r_\tau}   } \notag\\
  \le  &\  C'(r_\tau ,{s_\tau},n)(1+|s_0-s_1 |\, |t|)^{n/2+1} \Big\|
    |\varphi_j|^{r(\frac{1-z}{r_0}+\frac{z}{r_1})}e^{i \textup{Arg } (\varphi_j)}    \Big\|_{ L^{r_\tau}   } \notag\\
      \le  &\  C''(r_0, r_1,s_0,s_1 , \tau ,n)(1+ \, |t|)^{n/2+1} \Big\|
    |\varphi_j|^{r(\frac{1-\tau}{r_0}+\frac{\tau}{r_1})}    \Big\|_{ L^{r_\tau}   } \notag\\
          =  &\  C''(r_0, r_1,s_0,s_1 , \tau ,n)(1+ \, |t|)^{n/2+1} \big\|    \varphi_j     \big\|_{ L^{r }   } ^{\f r{r_\tau}}\, .\notag
\end{align}
It follows from this that
 \begin{equation}\label{200}
\| \si_{\tau+it} \|_{\li} \le  C''(r_0, r_1,s_0,s_1 , \tau ,n)(1+ \, |t|)^{n/2+1}
         \Big(  \sup_{j\in\mathbb Z}\big\|{\sigma(2^j\cdot)\widehat{\Psi}}\big\|_{L^{r}_{s}} \Big)^{\f r{r_\tau}} \, .
   \end{equation}

  Let $T_z$ be the family of operators associated to the multipliers $\sigma_z.$ Let $\ve $ be given.
  Fix $f, g\in  \mathscr C_0^\nf$ and $0<p_0<p<p_1<\nf$, $1<q_0'<q'<q_1'<\nf$.
  Given $\ve>0$, by Lemma \ref{lem:016} there exist functions $f_z $ and
  $ {g}_z$
  such that
  $  \norm{f_\theta^\ve-f}_{L^p}<\ve$, $ \norm{g_\theta^\ve-g}_{L^{q'}}<\ve, $ 
  and that
  \begin{align*}
 & \norm{f_{it}^\ve}_{L^{p_0}}\le   \big( \norm{f}_{L^{p}}+\ve\big)^{\frac p{p_0}} ,\quad
  \norm{f_{1+it}^\ve}_{L^{p_1}}\le   \big( \norm{f}_{L^{p}}+\ve\big)^{\frac p{p_1}},\\
 & \norm{g_{it}^\ve}_{L^{q_0'}}\le   \big( \norm{g}_{L^{q'}}+\ve\big)^{\frac {q'}{q_0'}},\quad
  \norm{g_{1+it}^\ve}_{L^{q_1'}}\le  \big( \norm{g}_{L^{q'}}+\ve\big)^{\frac {q'}{q_1'}}.
  \end{align*}
  Define
  {\allowdisplaybreaks
  \begin{align*}
  F(z) =& \int T_{\sigma_z}(f_z^\ve) {g}_z^\ve\; dx\\
  =&\int \sigma_z(\xi)\widehat{f_z^\ve}(\xi)\widehat{g_z^\ve}(\xi)\; d\xi\\
  =& \sum_{j\in \mathbb Z}\int (I-\Delta)^{-\frac{s_0(1-z)+s_1 z}{2}}
  \left[
  |\varphi_j|^{r(\frac{1-z}{r_0}+\frac{z}{r_1})}e^{i \textup{Arg } (\varphi_j)}
  \right](2^{-j}\xi)\widehat{\Phi}(2^{-j}\xi)
  \widehat{f_z^\ve}(\xi)\widehat{ {g}_z^\ve}(\xi)\; d\xi\\
  =& \sum_{j\in \mathbb Z}\int
  \left[|{\varphi_j}|^{r(\frac{1-z}{r_0}+\frac{z}{r_1})}e^{i \textup{Arg } (\varphi_j)}
  \right](2^{-j}\xi)(I-\Delta)^{-\frac{s_0(1-z)+s_1 z}{2}}
  \left[\widehat{\Phi}(2^{-j}\cdot)
  \widehat{f_z^\ve}\widehat{   {g}_z^\ve}\right](\xi)\; d\xi.
  \end{align*}
  }
  Notice that
  $ 
  (I-\Delta)^{-\frac{s_0(1-z)+s_1 z}{2}}
   [\widehat{\Phi}(2^{-j}\cdot)   \widehat{f_z^\ve}\widehat{   {g}_z^\ve}   ] (\xi)
  $ 
  is equal to a finite sum of  the form
  $$
 \sum_{k,l} |c_k^\ve|^{\frac p{p_0} + ( \frac p{p_1}-\frac p {p_0}) z}  |d_l^\ve|^{\frac{ q'}{q_0'} + ( \frac {q'}{q_1'}-\frac {q'} {q_0'}) z}
  (I-\Delta)^{-\frac{s_0(1-z)+s_1 z}{2}}
  \left[\widehat{\Phi}(2^{-j}\cdot)  \zeta_{k,l} \right]  (\xi) = H(\xi,z),
  $$
  where $\zeta_{k,l}$ are Schwartz functions, and thus it is an analytic function in $z$.

  Lemma \ref{lem:analint} guarantees that $F(z)$ is analytic on the strip $0<\Re(z)<1$ and continuous up to the boundary. Furthermore, by H\"older's inequality,
  $ 
  \abs{F(it)}\le \norm{T_{\sigma_{it}}(f_{it}^\ve)}_{L^{q_0}}\norm{g_{it}^\ve}_{L^{q_0'}},
  $ 
and
  \begin{align*}
  \norm{T_{\sigma_{it}}(f_{it}^\ve)}_{L^{q_0}} 
  \le& K_0\sup_{k\in\mathbb Z}
  \norm{\sigma_{it}(2^k\cdot)\widehat{\Psi}}_{L^{r_0}_{s_0}}\norm{f_{it}^\ve}_{L^{p_0}}\\
  \le&C(n,r_0) (1+|s_1-s_0|\, |t|)^{\frac{n}2+1}K_0\,
  \sup_{j\in\mathbb Z}\| {\varphi_j}\|_{L^r}^{\frac{r}{r_0}}
   \big( \norm{f}_{L^{p}}^p+\ve'\big)^{\frac 1{p_0}} \\
  = & C(n,r_0 )  (1+|s_1-s_0|\, |t|)^{\frac{n}2+1}K_0\,
  \sup_{j\in\mathbb Z}
   \big\|{(I-\Delta)^{\frac{s}2}[\sigma(2^j\cdot)\widehat{\Psi}]} \big\|_{L^r}^{\frac{r}{r_0}}
 \big( \norm{f}_{L^{p}}^p+\ve'\big)^{\frac 1{p_0}} .
  \end{align*}
  Thus, for some constant $C=C(n,r_0,s_0,s_1)$ we have
  $$
  \abs{F(it)}\le C  (1+\abs{t})^{\frac{n}2+1}K_0
  \sup_{j\in\mathbb Z}
   \big\|{(I-\Delta)^{\frac{s}2}[\sigma(2^j\cdot)\widehat{\Psi}]} \big\|_{L^r}^{\frac{r}{r_0}}
 \big( \norm{f}_{L^{p}}^p+\ve'\big)^{\frac 1{p_0}}
  \big( \norm{g}_{L^{q'}}^{q'}+\ve'\big)^{\frac {1}{q_0'}}.
  $$
  Similarly, for some constant $C=C(n,r_1,s_0,s_1)$ we   obtain
  $$
  \abs{F(1+it)}\le C  (1+\abs{t})^{\frac{n}2+1}K_1
  \sup_{j\in\mathbb Z}
   \big\|{(I-\Delta)^{\frac{s}2}[\sigma(2^j\cdot)\widehat{\Psi}]} \big\|_{L^r}^{\frac{r}{r_1}}
\big( \norm{f}_{L^{p}}^p+\ve'\big)^{\frac 1{p_1}}
  \big( \norm{g}_{L^{q'}}^{q'}+\ve'\big)^{\frac 1{q_1'}}
  $$
Thus for $z=\tau+it$,  $t\in\mathbb{R}$  and $0\le \tau \le  1$
it follows from \eqref{200} and from the definition of $F(z)$  that
$$
|F(z)|\le %A_\tau(t) \le
C'' %(r_0, r_1,s_0,s_1 , \tau ,n)
(1+ \, |t|)^{\f n2+1}
         \Big(  \sup_{j\in\mathbb Z}\big\|{\sigma(2^j\cdot)\widehat{\Psi}}\big\|_{L^{r}_{s}} \Big)^{\f r{r_\tau}}
        \|f_z^\ve\|_{L^2}   \|g_z^\ve\|_{L^2} 
=A_\tau(t)\, ,
$$
noting that $  \|f_z^\ve\|_{L^2}   \|g_z^\ve\|_{L^2} $ is bounded above by constants independent of $t$ and $\tau$. 
Since $A_\tau(t)\le
\exp(A e^{a|t|})  $
it follows that the hypotheses of Lemma~\ref{lem:ThreeLines} are valid.

  Applying Lemma~\ref{lem:ThreeLines} we obtain
  $$
  \abs{F(\theta)}\le C\, K_0^{1-\theta} K_1^{\theta}\sup_{j\in\mathbb Z}
  \big\|{(I-\Delta)^{\frac{s}2}[\sigma(2^j\cdot)\widehat{\Psi}]} \big\|_{L^r}
 \big( \norm{f}_{L^{p}}^p+\ve'\big)^{\f{1}{p}}
  \big( \norm{g}_{L^{q'}}^{q'}+\ve'\big)^{\f{1}{q'}} .
  $$

  But
  $$
  F(\theta) = \int_{\mathbb R^n} \sigma(\xi) \widehat{f_\theta^\ve}(\xi) \wh{ {g}_\theta^\ve}(\xi) \, d\xi \, .
  $$
 Then
\begin{align*}
 \bigg| \int_{\mathbb R^n} \sigma(\xi)& \widehat{f_\theta^\ve}(\xi) \wh{ {g}_\theta^\ve}(\xi) \, d\xi
- \int_{\mathbb R^n} \sigma(\xi) \widehat{f }(\xi) \wh{  g }(\xi) \, d\xi \bigg| \\
= &\q\bigg|  \int_{\mathbb R^n} \sigma(\xi) \Big[ \widehat{f_\theta^\ve}(\xi) \big(\wh{ {g}_\theta^\ve}(\xi)-\wh{  g }(\xi) \big)
 +  \widehat{g }(\xi) \big(\wh{ {f}_\theta^\ve}(\xi)-\wh{  f }(\xi) \big)  \Big]\, d\xi \bigg|  \\
 \le &\q \|\si\|_{\li}  \Big[ \|f_\theta^\ve\|_{L^2} \|{g}_\theta^\ve -g\|_{L^2} +
 \|g_\theta^\ve\|_{L^2} \|{f}_\theta^\ve -f\|_{L^2} \Big] \\
  \le &\q \|\si\|_{L^r_s}  \Big[ \|f \|_{L^2} \|{g}_\theta^\ve -g\|_{L^2} +
 \|g \|_{L^2} \|{f}_\theta^\ve -f\|_{L^2} \Big]\, .
\end{align*}
But the sequences ${f}_\theta^\ve -f$ and ${g}_\theta^\ve -g$ converge  to zero in $L^2$.
Letting $\ve\to 0$, these observations imply  that
$$
\bigg|\int_{\mathbb R^n} \sigma(\xi) \widehat{f }(\xi) \wh{  g }(\xi) \, d\xi \bigg|\le
C\, K_0^{1-\theta} K_1^\theta \sup_{j\in\mathbb Z}
  \big\|{(I-\Delta)^{\frac{s}2}[\sigma(2^j\cdot)\widehat{\Psi}]}\big\|_{L^r}
\norm{f}_{L^{p}} \|g\|_{L^{q'}}
$$
and   taking the supremum over all functions $g\in L^{q'}$ with $\| g\|_{L^{q'}} \le 1$ we obtain
  $$
  \norm{T_{\sigma }(f)}_{L^q}\le C_*\, K_0^{1-\theta} K_1^\theta \sup_{j\in\mathbb Z}
  \big\|{(I-\Delta)^{\frac{s}2}[\sigma(2^j\cdot)\widehat{\Psi}]}\big\|_{L^r}
\norm{f}_{L^{p}}, 
  $$
  where  $C_*=C_*(n,r_1,r_2,s_0,s_1)$. 
\end{proof}

\section{Proof of   Boundedness in Theorem~\ref{ThmMain}}

 To prove Theorem~\ref{ThmMain}
we use Theorem~\ref{interpL} applied as follows:  fix $p\in (1,2)$ such that $ \frac1p-\frac12 <\frac{s}{n} $.
Pick $p_0 =  1+\de $ with $\de$ small such that  $1<p_0<p$ and set $s_0=n/2+\ve$ and $r_0=2$ where $\ve$ is small.  Also set $p_1=2$, $s_1= \ve+\ve^2$, and $r_1= n/\ve  $.
We have that
  \begin{equation}\label{300}
    \norm{T_\sigma(f)}_{L^{p_0}}\le C(n,p_0,r_0,s_0)\sup_{j\in\mathbb Z}\norm{\sigma(2^j\cdot)\widehat{\Psi}}_{L^{r_{0}}_{s_0}}
    \norm{f}_{L^{p_0}}
  \end{equation}
  and
    \begin{equation}\label{400}
    \norm{T_\sigma(f)}_{L^{2}}\le C(n,p_0,r_1,s_1)
    \sup_{j\in\mathbb Z}\norm{\sigma(2^j\cdot)\widehat{\Psi}}_{L^{r_{1}}_{s_1}}
    \norm{f}_{L^{2}}
  \end{equation}
The conditions
$$
\frac 1p  =\frac{1-\theta}{p_0} + \frac{ \theta}{p_1} , \quad
       \frac 1r =\frac{1-\theta}{r_0} + \frac{ \theta}{r_1} , \quad
          s =(1-\theta)s_0 +   \theta s_1  .
  $$
translate into
  $$
 \f1 p -\f12 =\f{1 }{1+\de} -\f12+\theta \Big(\f12-\f1{1+\de}\Big)  , \q   \frac 1r =\frac{1-\theta}{2} + \frac{ \theta\ve}{n} , \quad s= (1-\theta) \f n2 +(1+\theta) (\ve+\ve^2)
  $$
  or
\begin{align*}
  \f1 p -\f12 &=\bigg( \f sn -(1+\theta) \f{ \ve+\ve^2} n  \bigg) \bigg(1-\f{2\de}{1+\de} \bigg)  \\
  &=\f sn -\bigg((1+\theta) \f{ \ve+\ve^2} n +\f{2\de}{1+\de}\f sn -(1+\theta) \f{\ve+\ve^2}{n} \f{2\de}{1+\de} \bigg) <\f sn. 
\end{align*}
Since $\de$ and $\ve$ are very small it follows that $\f 1p -\f12 $ can be arbitrarily close to $\f sn$. Note that once
$s$ is fixed for a given $p$, the optimal $r$ is close  to $\f ns$ (i.e., $\f 1r = \f sn-\f {\ve'} n$).
Interpolating between \eqref{300} and  \eqref{400}, via Theorem~\ref{interpL}, yields the required assertion.

\section{Necessary Conditions}

In this section we discuss examples that reinforce the minimality of the conditions on the indices in Theorem~\ref{ThmMain}.
One way to see this  is to use the multiplier
$m_{a,b}(\xi) = \psi(\xi) |\xi|^{- b} e^{i|\xi|^a}$
where $a > 0$, $a \neq   1$, $b > 0$, and $\psi$ is a smooth function which vanishes in a
neighborhood of the origin and is equal to $1$ for large $\xi$. One can verify that $m_{a,b}$
satisfies \eqref{nj} for $s = b/a$ and $r>n/s$. But it is known that $T_{m_{a,b}}$ is bounded
in $L^p(\mathbb R^n)$, $1 < p < \nf$, if and only if $   |\f 1p-\f 12|\le \f{b/a}{n} $ (see  Hirschman
 \cite[comments  after Theorem 3c]{hirschman2},    Wainger~\cite[Part II]{W}, and    Miyachi~\cite[Theorem 3]{Miy}).  Alternative examples were given in Miyachi and Tomita
\cite[Section 7]{MT}.

In this section we provide yet new examples to indicate the necessity of the indices in   Theorem~\ref{ThmMain}.
We are not sure as to whether boundedness into $L^p$, or even weak $L^p$, is valid in general
under  assumption \eqref{2}   exactly on the critical line
$\big|\frac1p-\frac12\big|=\frac{s}{n} $.

\begin{prop}\label{Nec}
If for all $\si \in \li(\rn)$ such that $\sup_k\|\si(2^k\cdot)\wh\Psi\|_{L^r_s(\bbr^n)}<\nf$ we have
\begin{equation}\label{MKI}
\|T_\si\|_{L^p(\rn)\to L^p(\rn)} \le C_p \, \sup_k\|\si(2^k\cdot)\wh\Psi\|_{L^r_s(\bbr^n)} <\nf\, ,
\end{equation}
then we must necessarily have $rs\ge n$ and $|\tf1p-\tf12|  \le   \tf sn$.
\end{prop}

\begin{proof}
First we prove the necessary condition $rs\ge n$. Let $\wh\zeta$ be a smooth function supported in the ball $B(0,1/10)$ in $\rn$ and let $\wh\phi$ be
 supported in the ball $B(0,1/2)$ equal to $1$ on  $B(0,1/5)$.
Define $\wh f (\xi )=\wh\zeta(N(\xi -a))$ with $|a|=1$, and $\si(\xi )= \wh
\phi(N(\xi -a))$, then a direct calculation gives
$\|f \|_{L^{p }(\rn)}\approx N^{-n+n/p }$ and $\|\si\|_{L^r_s(\bbr^{ n})}
\le CN^sN^{- n/r}$; for the last estimate see Lemma~\ref{Sob}.
Moreover, $T_{\si}(f )(x)=N^{- n} \zeta(x/N)e^{2\pi ix\cdot a} $.
We thus   obtain that
$\|T_\si(f )\|_{L^p(\rn)}\approx N^{- n+n/p}  $.
Then \eqref{MKI} yields the inequality $N^{- n+n/p}\le CN^sN^{- n/r} N^{-n+n/p }$,
which forces $s- n/r\ge0$ by letting $N$ go to infinity.

 We now turn to the other necessary condition $|\tf1p-\tf12|  \le   \tf sn$.
By duality it suffices to prove the case when $1<p\le 2$. We will prove our result by
constructing an example. We consider the case $n=1$ first while the higher dimensional
case will be an easy generalization.

Let $\wh \psi, \wh\vp\in \mathscr C^{\nf}_0(\mathbb R)$ such that $0\le \wh\vp\le\chi_{[-1/100,1/100]}$ and
$\chi_{[-1/10,1/10]}\le \wh\psi\le \chi_{[-1/2,1/2]}$.
Therefore $\wh\psi\wh\vp=\wh\vp$.
For a fixed large positive integer $N$, we define
\begin{equation}\label{DefFN}
\wh f_N(\xi)=\sum_{j=-N}^N\wh\vp(N\xi-j) , \qq
 \si_{N,\, t}(\xi)=\sum_{j\in J_N}a_j(t)\wh\psi(N\xi-j) ,
\end{equation}
 where $J_N=\{j\in\mathbb Z:\ \tf N2\le|j|\le 2N\}$ and $t\in [0,1]$.
Here $\{a_j\}_{j=-\nf}^{\nf}$ is the sequence of Rademacher functions indexed by all integers.

One can verify that
$T_{N,\, t}(f_N)=:(\si_{N,\, t}f_N)^{\vee}=(\sum_{j\in J_N}a_j(t)\wh\vp(N\xi-j))^{\vee}$. Recall
that Rademacher functions satisfy   for any $p\in(0,\nf)$
\begin{equation*}
c_p\Big\|\sum_ja_j(t)A_j\Big\|_{L^p([0,1])}\le \Big(\sum_j|A_j|^2\Big)^{1/2}\le C_p\Big\|\sum_ja_j(t)A_j\Big\|_{L^p([0,1])},
\end{equation*}
where   $c_p$ and $C_p$ are constants.
Therefore
\begin{align*}
\bigg(\int_0^1 \|T_{N,\, t}(f_N)\|_{L^p(\mathbb R)}^pdt \bigg)^{1/p}=&
\bigg(\int_0^1\int_{\mathbb R} \bigg| \sum_{j\in J_N}a_j(t)N^{-1}\vp(N^{-1}x)e^{2\pi ixj/N}\bigg|^pdxdt\bigg)^{1/p}\\
\approx &
\bigg(\int_{\mathbb R}\bigg( \sum_{j\in J_N} \Big|N^{-1}\vp(N^{-1}x)e^{2\pi ixj/N}\Big|^2\bigg)^{p/2}dx\bigg)^{1/p}\\
\approx&N^{-1}
\Big(\int_{\mathbb R}|N^{1/2}\vp(N^{-1}x)|^{p}dx\Big)^{1/p}\\
\approx &N^{1/p-1/2}.
\end{align*}

The Sobolev norm of $\si_{N,t}$ is given by the following lemma,      proved in all dimensions.
For $n\ge 1$ and $\vec t=(t_1,\dots, t_n)\in [0,1]^n$ we define a function on $\mathbb R^n$ by
$$
\si_{N,\,\vec t\,}(x_1,\dots,x_n)=\sum_{\vec j\in J_N}a_{j_1}(t_1)\cdots a_{j_n}(t_n)\wh\vp(N\xi_1-j_1)
\cdots \wh\vp(N\xi_n-j_n) \, ,
$$
where  $J_N=\{\vec j=(j_1,\dots,j_n)\in\mathbb Z^n:
\tf N2\le |\vec j|\le 2N\}$. This $\si_{N,\,\vec t}$ coincides with $\si_{N,\, t}$ when $n=1$.

\begin{lem}\label{Sob}%Sobolev Norm
We have that
$
\|\si_{N,\,\vec t\,} \|_{L^r_s(\mathbb R^n)}\le C N^{s}.
$
\end{lem}

We postpone the proof of the lemma and continue with the proof of Proposition~\ref{Nec} when $n=1$.
We note that $\wh f_N$ has $L^q$ norm bounded by a constant independent of $N$,
which implies by the Young's inequality that $\|f_N\|_{L^q}\le C$
with $C$ independent of $N$ when $2\le q\le\nf$. We show in the following lemma
that this property is valid for all $q\in (1,\nf]$.

\begin{lem}\label{Fun}%Function
Let $f_N$ be as in \eqref{DefFN} and let $p\in (1,\nf]$. Then
there is  a constant $C_p$ independent of $N$
such that $\|f_N\|_{L^p}\le C_p$.
\end{lem}

\begin{proof}
We note that
$f_N=\sum_{j=-N}^N\f1N\vp(x/N)e^{2\pi ixj/N}=\tf1N\vp(x/N)D_N(x/N)$,
where $D_N$ is the Dirichlet kernel, whose $L^p$-norm over $[0,1]$ is comparable to
$N^{1/p'}$ when $p>1$; see for example \cite[Exercise 3.1.6]{CFA}. Using this fact and     that
$\vp$ is a Schwartz function we obtain
\begin{align*}
\|f_N \|_{L^p(\mathbb R)}=& \Big\|\tf1N\vp(\tf\cdot N)D_N(\tf\cdot N)\Big\|_{L^p(\mathbb R)}\\
=&\tf 1NN^{1/p}\|\vp D_N\|_{L^p(\mathbb R)}\\
=&N^{-1/p'}\bigg(\sum_{j=-\nf}^{\nf}\int_{j-1}^j|\vp(x)D_N(x)|^pdx\bigg)^{1/p}\\
\le &CN^{-1/p'}\bigg(\sum_{j=-\nf}^{\nf}\f1{(1+|j|)^M}\int_{j-1}^j|D_N(x)|^pdx\bigg)^{1/p}\\
\le &C_p N^{-1/p'}N^{1/p'}=C_p\, .
\end{align*}
This proves the claim.
\end{proof}

In view of Lemma~\ref{Fun}  we   obtain the following inequalities
$$
N^{\f 1p-\f12 }\le  C\bigg(\int_0^1\|T_{N,\, t}(f_N)\|_{L^p(\mathbb R)}^pdt\bigg)^{\f1p}
\le  C\, A\,  \|f\|_p\bigg(\int_0^1 \|\si_{N,\, t}\|_{L^r_s} ^p dt\bigg)^{\f1p}
\le   C\, C_pAN^s.
$$
Letting $N$ go to infinity forces $1/p-1/2\le s$.

We now consider the higher dimensional case. Let $F_N(\vec x\,)=f_N(x_1)\cdots f_N(x_n)$, where $f_N$ is as in \eqref{DefFN}.
It follows from Lemma \ref{Sob} and \ref{Fun}
that $\|F_N\|_{L^p}\le C$ and $\|\si_N\|_{L^r_s}\le CN^s$. A calculation similar to the
one dimensional case shows that $\|T_N(F_N)\|_{L^p}\approx N^{(1/p-1/2)n}$, thus letting $N\to \nf$ we obtain
  that $|1/p-1/2|\le s/n$.
\end{proof}

We now prove
Lemma \ref{Sob}.
\begin{proof}[Proof of Lemma \ref{Sob}]
It is easy to verify that $\| \si_{N,\,\vec t\,}\|_{L^r}\le C$
and $\|\si_{N,\, \vec t\,}\|_{L^r_2}\le CN^2$.
Define for $z=u+iv$ and $\phi\in\mathscr S(\bbr^{n})$
the   function
$$
F(z)=\int_{\bbr^{n}}(I-\Delta)^{z}\si_{N,\,\vec t\,}(x)\phi(x)dx
$$
for $z$ in the closed unit strip. Then $F$ is analytic on the open strip and continuous on its closure.
 We can also show that by the Mihlin
multiplier theorem that
$$
|F(z)|\le P(|v|)\|\si_{N,\,\vec t}\|_{L^r_{2u}}\|\phi\|_{L^{r'}}\le P(|v|)N^{2u+n/r'}\|\vp\|_{L^{r}_{2u}}\|\phi\|_{L^{r'}},
$$
where $P(t)$ is a polynomial in $t$ which is not necessary to be the same at all occurrences.

We have then $\log|F(z)|\le \log(N^{2u+n/r'}\|\vp\|_{L^r_{2u}}\|\phi\|_{L^{r'}})+C\log|v|\le Ce^{\tau_0|v|}$
for some $\tau_0\in (0,1)$. Applying Lemma~\ref{lem:ThreeLines} we obtain for $0<s<1$
that
\begin{equation}\label{SN1}
\log |F(s)|\le \f{\sin(\pi s)}2\int_{-\nf}^{\nf} \bigg[\f{M_0(t)}{\cosh(\pi t)-\cos(\pi s)}+\f{M_1(t)}{\cosh(\pi t)+
\cos(\pi s)}\bigg]dt,
\end{equation}
where $\log |F(it)|\le M_0(t)=c\log |t|+\log \|\vp\|_{L^{r'}}$
and $\log |F(1+it)|\le M_1(t)=2\log N+c\log |t|+\log \|\vp\|_{L^{r'}}$.

We show that  \eqref{SN1} is controlled by $2s\log N+C(s)+\log \|\vp\|_{L^{r'}}$,
where
$C(s)$ is a finite constant depending on $s$ and independent on $N$. Then
$$
|F(s)|\le e^{2s\log N}e^{C(s)}e^{\log\|\vp\|_{L^{r'}}}\le C(s)N^{2s}\|\vp\|_{L^{r'}} , 
$$
i.e.
$|F(s/2)|\le C(s)N^s\|\phi\|_{L^{r'}}$ for all $\phi\in\mathcal S$, hence
$\|(I-\Delta)^{s/2}m\|_{L^r}\le C(s)N^s$ for $s\in (0,2)$. Note that the original
estimate
$\|\si_{N,\, \vec t}\|_{L^r_m}\le CN^m$ is valid for any positive integer $m$, so a similar argument gives the
estimate $\|\si_{N,\,\vec t}\|_{L^r_s}\le CN^s$ for all $s\ge0$.

It remains to control \eqref{SN1}, for which we recall \eqref{ide}.
So
\begin{align*}
\f{\sin(\pi s)}2\int_{-\nf}^{\nf} \bigg[\f{\log\|\vp\|_{L^{r'}}}{\cosh(\pi t)-\cos(\pi s)}+
\f{\log\|\vp\|_{L^{r'}}+2\log N}{\cosh(\pi t)+\cos(\pi s)}\bigg]\, dt  %\\
%& \qq\qq\qq\qq\qq
=\log\|\vp\|_{L^{r'}}+2s\log N.
\end{align*}
So matters reduce to showing that for $0<s<1$ we have
$$
\int_{-\nf}^{\nf} \f{\log|t|}{\cosh(\pi t)-\cos(\pi s)}dt+\int_{-\nf}^{\nf} \f{\log|t|}{\cosh(\pi t)+\cos(\pi s)}dt<\nf
$$
which is a straightforward calculation.
 \end{proof}

 \begin{comment}
 We prove that second one while the first one is similar.

$\int_{-\nf}^{0}\f{\log|t|}{\cosh( t)+\cos( s)}dt<\nf$ for $0<s<\pi$. Let $v=e^t$,
then
\begin{align*}
\int_{-\nf}^{0}\f{\log|t|}{\cosh( t)+\cos( s)}dt=&2\int_0^\nf\f{\log t}{e^t+e^{-t}+2\cos s}dt\\
=&2\int_0^\nf\f{\log\log v}{v^2+2v\cos s+1}dv\\
=&2\int_1^e\f{\log\log v}{v^2+2v\cos s+1}dv+2\int_e^\nf\f{\log\log v}{v^2+2v\cos s+1}dv\\
=:&I+II
\end{align*}
We observe that $|\log t|\le Ct^{-\ep}$ for $0<t<1$, hence
 $I\le \int_1^e\f{C(\log v)^{-\ep}}{2+2\cos s}dv\le C(2+2\cos s)^{-1}<\nf$.
We can control $II$ by $C\int_e^{\nf}\f{v^\ep}{(v+\cos)^2+\sin^2s}dv<\nf$.
\end{comment}

\section{The endpoint case $\big|\f1p-\f12 \big|= \f sn$}

As another application of the interpolation  technique of this paper, we discuss an interpolation theorem  applicable in the critical case   
$\big|\f1p-\f12 \big|= \f sn$. 
We introduce the  Besov space   norm
$$
 \| h \|_{B_{p,q }^s}: = \bigg( \sum_{j\ge 1}   \big\| 2^{js} \De_j h \|_{\lp}^q \bigg)^{\f 1q} + \big\|S_0 h\big\|_{L^p}
$$
where $\De_j$ are the Littlewood-Paley operators and $S_0$ is an averaging operator that satisfy
$S_0+\sum_{j=1}^\nf \De_j  = I$. We assume that for $j\ge 1$, $\De_j$ have spectra supported in the annuli
$2^j \le |\xi | \le 2^{j+2}$, while $S_0$ has spectrum inside the ball $B(0,2)$.

We recall the following result of Seeger~\cite{Seeger} 
\beq\label{S1}
\big\| T_\si \big\|_{H^1\to L^{1,2}} \le C\, \sup_{k\in \mathbb Z}  \| \si (2^k \,\cdot\,) \wh\Psi \|_{B_{2}^{\f n2,1}}
\eeq
concerning the endpoint case $p=1$.
We also have the trivial estimate
\beq\label{S2}
\big\| T_\si \big\|_{L^2\to L^2}=\big\| T_\si \big\|_{L^2\to L^{2,2}} \le C\, \sup_{k\in \mathbb Z}  \| \si (2^k \,\cdot\,) \wh\Psi \|_{B_{ \nf }^{0,1}}\, .
\eeq

In this section,  we derive the intermediate estimate contained in Seeger~\cite{Seeger2}:
\beq\label{S3}
\big\| T_\si \big\|_{L^p\to L^{p,2}} \le C\, \sup_{k\in \mathbb Z}  \| \si (2^k \,\cdot\,) \wh\Psi \|_{B_{\f ns}^{s,1} }
\eeq
for $\big|\f1p-\f12 \big|= \f sn$,    $1<p<2$, and $0\le s\le \f n2$.
We deduce estimate \eqref{S3} from  the following theorem.
\begin{thm}\label{interpL2}
 Fix  $1< r_0,r_1\le \infty$, $1<p_0,p_1<\infty$, $0\le s_0,s_1<\infty$.
 Let $\wh{\Psi}$ be supported in the annulus $1/2\le |\xi|\le 2$ on $\mathbb R^n$ and satisfy
 $$
 \sum_{j\in \mathbb Z} \widehat{\Psi}(2^{-j}\xi) = 1,  \qq\xi\neq 0.
 $$
 Assume that for $k\in\set{0,1}$ we have
  \begin{equation}\label{Assump}
    \norm{T_\sigma(f)}_{L^{p_k,2}}\le K_k\sup_{j\in\mathbb Z}\big\|{\sigma(2^j\cdot)\widehat{\Psi}}\big\|_{B^{s_k,1}_{ r_k}}
    \norm{f}_{L^{p_k }}
  \end{equation}
  for   all $f\in \mathscr C_0^\nf(\mathbb R^n)$ and $\sigma$ which make the right hand side finite.
For $0<\theta<1$ define
  $$
  \frac 1p  =\frac{1-\theta}{p_0} + \frac{ \theta}{p_1} , \quad
%    \frac 1q =\frac{1-\theta}{q_0} + \frac{ \theta}{q_1} , \quad
       \frac 1r =\frac{1-\theta}{r_0} + \frac{ \theta}{r_1} , \quad
          s =(1-\theta)s_0 +   \theta s_1  .
  $$
Then  there is  a constant $C_*= C_*( r_0,r_1,s_0,s_1, p_0,p_1, p,n) $ such that for all $f$ in $ \mathscr C_0^\nf(\mathbb R^n)$
 we have
  \begin{equation}\label{Conc}
    \norm{T_\sigma(f)}_{L^{p,2}(\rn)}\le C_* K_0^{1-\theta} K_1^{\theta} \sup_{j\in\mathbb Z}
    \big\|{\sigma(2^j\cdot)\widehat{\Psi}}\big\|_{B^{s,1}_{ r}}
    \norm{f}_{L^{p}(\rn)}.
  \end{equation}
Moreover,     conclusion \eqref{Conc} also holds under the assumption that $p_0=1$ and \eqref{Assump}  is
substituted (only for    $k=0$) by
  \begin{equation}\label{Assump2}
    \norm{T_\sigma(f)}_{L^{1,2}}\le K_0\sup_{j\in\mathbb Z}\big\|{\sigma(2^j\cdot)\widehat{\Psi}}\big\|_{B^{s_0,1}_{r_0}}
    \norm{f}_{H^1}
  \end{equation}
for all $f\in \mathscr C_0^\nf(\mathbb R^n)$ with vanishing integral.
\end{thm}

\begin{proof}
Let $\wh{\Phi}(\xi ) = \sum_{j\le 0}  \widehat{\Psi}(2^{-j}\xi)$ and $\wh{\Phi}(0)=1$; then
$\wh{\Phi} $ is supported in $|\xi|\le 2$.
Fix a bounded function $\si$. For an integer $k$ define the dilation of $\si^k$ by setting $\si^k(\xi) = \si(2^k \xi)$.
For $z$ in the closed unit strip we introduce linear functions
$$
L(z) = \f{r}{r_0} (1-z) +\f r {r_1} z , \qq M(z) = s- (1-z)s_0-zs_1
$$
and  when $j\ge 1$  introduce
Littlewood-Paley operators $\De_j (g) = g* \Psi_{2^{-j}}$, $\wt{\De}_j  (g) = g* \wt{\Psi}_{2^{-j}}$,   where
$\wt \Psi$ is a Schwartz function whose Fourier transform is supported in an annulus only slightly larger than
$1/2\le |\xi|\le 2$ and equals $1$ on the support of $\wh{\Psi}$.
We also define $\De_0(g) = g*\Phi$ and $\wt{\De}_0(g) = g*\wt{\Phi}$, where  the Fourier transform
of $\wt{\Phi}$ is supported in $|\xi|\le 4$ and equals $1$ on the support
of $\wh{\Phi}$. Then define:
 
$$
\si_z = \sum_{k\in \mathbb Z} \sum_{j= 0}^\infty 2^{j M(z)} (c_j^k)^{1-L(z)} \,\, \wt{\De}_j \Big(
\big| \De_j (\si^k \wh{\Psi} ) \big|^{L(z)} e^{i \textup{Arg} \big( \De_j (\si^k \wh{\Psi} )  \big)} \Big) (2^{-k} \cdot )
\wh{\wt{\Psi}}  (2^{-k} \cdot )
$$
where
$$
c_j^k = \f{\big\| \De_j (\si^k \wh{\Psi}) \big\|_{L^r} }{
\sup\limits_{\mu\in \mathbb Z} \sum\limits_{l\ge 0}  2^{ls}\big\| \De_l  (\si^\mu \wh{\Psi}) \big\|_{L^r}} \, .
$$
Next, we estimate
\begin{equation}\label{nhji}
\sup_{\mu\in \mathbb Z} \sum_{l \ge 0} 2^{ls_0} \big\| \De_l \big(\si_{it}^\mu \wh{\Psi} \big)\big\|_{L^{r_0}}\, .
\end{equation}

We notice that for a given $\mu\in \mathbb Z$, in the sum defining $\si_{it}^\mu$, only finitely many terms in $k$ appear,
the ones with $k=\mu, \mu+1,\mu-1$. For simplicity we only consider the term with $k =\mu$, since the other ones are similar. 
This part of \eqref{nhji} is estimated by
\begin{equation}\label{nhji2}
\sup_{\mu\in \mathbb Z} \sum_{l \ge 0} \sum_{j= 0}^N  2^{ls_0} 2^{j(s-s_0)}
|c_j^\mu|^{1-\f r{r_0}} \bigg\| \wt{\De}_l\bigg(
\wt{\De}_j \Big(
\big| \De_j (\si^\mu \wh{\Psi} ) \big|^{L(it)} e^{i \textup{Arg} \big( \De_j (\si^\mu \wh{\Psi} )  \big)} \Big) \wh{\Psi} \wh{\wt{\Psi}}  \bigg)
\bigg\|_{L^{r_0}}\, .
\end{equation}
Using Lemma~\ref{lemlast} (stated and proved below) we obtain that \eqref{nhji2} is bounded by
\begin{equation}\label{nhji3}
\sup_{\mu\in \mathbb Z} \sum_{l \ge 0} \sum_{j= 0}^N  2^{ls_0} 2^{j(s-s_0)}
|c_j^\mu|^{1-\f r{r_0}} C_M 2^{-2|1-\f 1{r_0}|\max(j,\, l)M}   \big\| \,  \big|
  \De_j (\si^\mu \wh{\Psi} ) \big|^{\f{r}{r_0}}  \big\|_{L^{r_0}}\, .
\end{equation}
But the sum over $l$ in \eqref{nhji3} is bounded by  $C_M 2^{js_0} 2^{-j2|1-\f 1{r_0}|M} \le C_M 2^{js_0}$ for $M$ sufficiently large, and consequently
\eqref{nhji3} is bounded by
\begin{equation}\label{nhji4}
C_M \sup_{\mu\in \mathbb Z}  
\sum_{j= 0}^\infty    \! 2^{j(s-s_0)} 2^{js_0}
|c_j^\mu|^{1-\f r{r_0}}     \big\| \De_j (\si^\mu \wh{\Psi} )   \big\|_{L^{r}}^{\f{r}{r_0}}
 \le C_M
 \bigg(\! \sup\limits_{\mu\in \mathbb Z} \sum\limits_{j\ge 0}\!  2^{js}\big\| \De_j  (\si^\mu \wh{\Psi}) \big\|_{L^r}\!\!\bigg)^{\!\!\f{r}{r_0}}
\end{equation}
by the choice of $c_j^\mu$.
Likewise we obtain a similar estimate for the point $1+it$. We summarize these two estimates   as follows:
\begin{equation}\label{nhji5}
\sup_{\mu\in \mathbb Z} \sum_{l \ge 0} 2^{ls_m} \big\| \De_l \big(\si_{z}^\mu \wh{\Psi} \big)\big\|_{L^{r_m}}
 \le C_M
 \bigg(\! \sup\limits_{k\in \mathbb Z} \sum\limits_{j\ge 0}  2^{js}\big\| \De_j  (\si^k \wh{\Psi}) \big\|_{L^r}\bigg)^{\!\!\f{r}{r_m}}
\end{equation}
where $m=0,1$ and $\Re z= m$.

Now consider an analytic family of operators $T_z$  associated with the multipliers $\sigma_z $  defined by 
$f\mapsto T_{\sigma_z}(f)$.
We have that when $\real z=0$, $T_z$ maps $L^{p_0,2}$ to $L^{p_0}$ if $p_0>1$ and 
$H^1$ to $L^{1,2}$   if $p_0=1$ 
with constant
$$
B_0= C_M K_0\bigg( \sup\limits_{k\in \mathbb Z} \sum\limits_{j\ge 0}  2^{js}\big\| \De_j  (\si^k \wh{\Psi}) \big\|_{L^r}\bigg)^{\f{r}{r_0}}
$$
and   when $\real z=1$, $T_z$  maps $L^{p_1,2}$ to $L^{p_1}$ with constant
$$
B_1=C_M K_1 \bigg( \sup\limits_{k\in \mathbb Z} \sum\limits_{j\ge 0}  2^{js}\big\| \De_j  (\si^k \wh{\Psi}) \big\|_{L^r}\bigg)^{\f{r}{r_1}}\, . 
$$
We now interpolate using Theorem 1.1 (with $m=1$) in \cite{grafakosmastylo}. We obtain 
$$
\| T_{\si_\theta}(f) \|_{ ( L^{p_0,2})^{1-\theta} (L^{p_1,2})^\theta} \le C(p_0,p_1,p)   B_0^{1-\theta} B_1^{\theta} 
\| f \|_{( L^{p_0 }, L^{p_1 })_\theta} \, . 
$$
Noting that $( L^{p_0,2})^{1-\theta} (L^{p_1,2})^\theta= L^{p,2}$ and $( L^{p_0 }, L^{p_1 })_\theta =L^p$ (even when $p_0=1$, in which case $L^{p_0}$ is  replaced by $H^1$), we obtain the claimed assertion. 
\end{proof}

\begin{lem}\label{lemlast}
Using the notation of Theorem~\ref{interpL2}, for any $M>0$ there is a constant $C_M$ (also depending on the dimension $n$, on $\Psi$, and $\wt\Psi$) such that for any $1\le q\le \nf$ we have
\beq\label{ppook}
\Big\| \wt{\De}_l\big(    \wt{\De}_j \big(  g \big) \wh{\Psi} \wh{\wt{\Psi}}  \big) \Big\|_{L^q } \le C_M 
{ 2^{-2(1-\f 1q)
\max(j,\, l)M}} \|g\|_{L^q}
\eeq
for all $l,j> 0$.  We also have that for any $M>n$ there is a constant $C_M$ such that
\beq\label{ppH1}
\Big\| \wt{\De}_l\big(    \wt{\De}_j \big(  g \big) \wh{\Psi} \wh{\wt{\Psi}}  \big) \Big\|_{L^1 } \le C_M 2^{- \max(j,\, l)(M-n)} \|g\|_{H^1}
\eeq
\end{lem}

\begin{proof}
The claimed estimate is obviously true when $q=1$. So we prove it for $q=2$ and derive \eqref{ppook} as
a consequence of classical Riesz-Thorin interpolation theorem. Examining the Fourier transform of the operator in
\eqref{ppook}, matters reduce to computing the $\li$ norm of the function
\beq\label{ppook2}
\wh{\wt \Psi}(2^{-j } \xi) \int_{\mathbb R^n} \wh{\wt\Psi}(2^{-l } (\xi-\eta) ) \phi(\eta) d\eta
\eeq
where $\phi(\eta) = \Psi * \wt\Psi$ is a Schwartz function. Since the integral is over the set $ |\xi-\cdot| \approx 2^l$, we estimate the absolute value of the expression in \eqref{ppook2} by
$$
C_M\Big[\sup\Big\{ \f{1}{(1+|\eta|)^M}:\,\,    |\xi-\eta|   \approx 2^l
\Big\} \Big] \,\, \int_{\mathbb R^n} (1+|\eta|)^{-M}d\eta
$$
where $  |\xi| \approx   2^j$. Notice that if $l> j+10$, then $|\eta|\approx 2^l$, while if $j>l+10$, then
$|\eta|\approx 2^j$. These estimates yield the proof of \eqref{ppook}.

We now turn our attention to \eqref{ppH1}. Using Fourier inversion, we write
$$
\wt{\De}_l\big(    \wt{\De}_j \big(  g \big) \wh{\phi}  \, \big)(x) =
\int_{\mathbb R^n} \wh{g}(\eta) \wh{\wt{\Psi}}(2^{-l}\eta)\int_{\mathbb R^n}
\wh{\wt{\Psi}}(2^{-j}\xi) \phi(\xi-\eta) e^{2\pi i x\cdot \xi} \, d\xi \, d\eta \, .
$$
We integrate by parts in the inner integral with respect to the operator $(I-\De_\xi)^{N}$ to obtain that the
preceding expression is equal to
$$
\sum_{\be+\ga=2N}\f{C_{\be,\ga} }{(1+4\pi^2|x|^2)^N}
\int_{\mathbb R^n} \wh{g}(\eta) \wh{\wt{\Psi}}(2^{-l}\eta)\int_{\mathbb R^n}
2^{-j|\be|} (\p^\be\wh{\wt{\Psi}})(2^{-j}\xi) (\p^\ga \phi)(\xi-\eta) e^{2\pi i x\cdot \xi} \, d\xi \, d\eta \, .
$$
Since for $g\in H^1$ we have $|\wh{g} (\xi) | \le c \|g\|_{H^1}$ for all $\xi$ and we deduce the estimate
$$
\big| \wt{\De}_l\big(    \wt{\De}_j \big(  g \big) \wh{\phi}  \, \big)(x) \big| \le
 \f{  C_M  \|g\|_{H^1} }{(1+4\pi^2|x|^2)^N} 2^{ln} \sup_{|\eta| \approx 2^l}
\int_{|\xi| \approx 2^j} \f{d\xi}{ (1+|\xi-\eta|)^{2M}}
$$
for $M>n$. We easily derive from this estimate the validity of \eqref{ppH1}. Note that in the case $j=0$   the notation 
$|\xi| \approx 2^j$ should be interpreted as $|\xi| \lesssim 2$; likewise when $l=0$.
\end{proof}

Note that Proposition~\ref{KKLLMM} is a consequence   of Theorem~\ref{interpL2} with initial   estimates \eqref{S1} and \eqref{S2}.

\medskip

\noindent{\bf Acknowledgment:} The first author would like to thank Andreas Seeger for    useful   discussions.

 \end{document}